\documentclass{article} 

\usepackage{amsmath,amsthm,amssymb,amsfonts}

\usepackage{geometry} 
\newtheorem{Theorem}{Theorem}[section]
\theoremstyle{plain}

\newtheorem{Definition}[Theorem]{Definition}

\newtheorem{Corollary}[Theorem]{Corollary}
\newtheorem{Proposition}[Theorem]{Proposition}

\numberwithin{equation}{section}

\title{A note on the polynomial moments of the partition function in the SK model}
\author{Sergey Bocharov}
\begin{document}
\maketitle
\begin{abstract}
We prove a simple identity relating the $k$th moment of the partition function $Z_N(\cdot)$ in the SK model 
to the $N$th moment of the partition function $Z_k(\cdot)$. As a corollary we find a characterisation of the limit 
$\lim_{N \to \infty} \frac{1}{N} \log \mathbb{E} Z_N(\beta)^k$ alternative to the one found previously by Michel Talagrand in \cite{4}.
\end{abstract}
\section{Introduction and Main Results}
For the Sherrington-Kirkpatrick model we are given a parameter $\beta > 0$ referred to as the inverse temperature, an integer $N \geq 2$ for 
the number of $\pm 1$ spins, vectors $\sigma = (\sigma_1, \cdots, \sigma_N) \in \{-1, 1\}^N$ representing different configurations of $N$ spins 
and independent standard normal random variables $g_{i j }$, $1 \leq i < j \leq N$, which describe the interactions between the spins. 

The partition function for the Sherrington-Kirkpatrick model is defined as 
\[
Z_N(\beta) := \sum_{\sigma \in \{-1, 1\}^N} \exp \Big\{ \frac{\beta}{\sqrt{N}} \sum_{1 \leq i < j \leq N} g_{i j} \sigma_i \sigma_j \Big\} \text{.}
\]
One interesting problem related to $Z_N(\beta)$ is to study the asymptotic behaviour of its various moments as $N \to \infty$. In particular the famous Parisi formula 
proved by M. Talagrand in \cite{3} characterises $\lim_{N \to \infty} \frac{1}{N} \mathbb{E} \log Z_N(\beta)$ as a solution of a certain deterministic optimisation problem. In his later paper 
\cite{4} Talagrand has also given a similar representation of $\lim_{N \to \infty} \frac{1}{N} \log \mathbb{E} Z_N(\beta)^a$ for all $a \in \mathbb{R}$, which we shall discuss in Section 4. 

In this paper we shall study $\mathbb{E} Z_N(\beta)^k$, $k \in \mathbb{N}$ proving an identity which relates $\mathbb{E} Z_N(\cdot)^k$ to 
$\mathbb{E} Z_k(\cdot)^N$. As a corollary we get a characterisation of $\lim_{N \to \infty} \frac{1}{N} \log \mathbb{E} Z_N(\beta)^k$ alternative to that given in 
\cite{4}.

Let us now present our main results. Firstly we claim that $k$th moments of the partition function satisfy the following relation. 
\begin{Proposition}
\label{prop}
For $k \geq 2$
\begin{equation}
\label{moments}
\mathbb{E} \Big[ Z_N(\beta) \Big]^k = \mathrm{e}^{\frac{\beta^2 k}{4} N - \frac{\beta^2 k^2}{4}} \mathbb{E} \Big[ Z_k(\beta \sqrt{\frac{k}{N}}) \Big]^N
\end{equation}
Or, equivalently,
\begin{equation}
\label{moments2}
\mathbb{E} \Big[ \mathrm{e}^{\frac{\beta^2}{4}} Z_N(\frac{\beta}{\sqrt{k}}) \Big]^k = 
\mathbb{E} \Big[ \mathrm{e}^{\frac{\beta^2}{4}} Z_k(\frac{\beta}{\sqrt{N}}) \Big]^N \text{.}
\end{equation}
\end{Proposition}
So, for example 
\begin{align*}
\mathbb{E} \Big[ Z_N(\beta) \Big]^2 &= \mathrm{e}^{\frac{\beta^2}{2}N - \beta^2} 4^N \mathbb{E} \Big( \cosh \big(\frac{\beta \sqrt{2}}{\sqrt{N}} g_{12}\big)\Big)^N \text{,}\\
\mathbb{E} \Big[ Z_N(\beta) \Big]^3 &= \mathrm{e}^{\frac{3}{4}\beta^2 N - \frac{9}{4} \beta^2} 2^N \mathbb{E} \Big( \big(
\mathrm{e}^{\frac{\beta \sqrt{3}}{\sqrt{N}} (g_{12} + g_{13} + g_{23})} + \mathrm{e}^{\frac{\beta \sqrt{3}}{\sqrt{N}} (g_{12} - g_{13} - g_{23})}\\ 
& \qquad\qquad\qquad\qquad + \mathrm{e}^{\frac{\beta \sqrt{3}}{\sqrt{N}} (-g_{12} - g_{13} + g_{23})} + 
\mathrm{e}^{\frac{\beta \sqrt{3}}{\sqrt{N}} (-g_{12} + g_{13} - g_{23})}\big)\Big)^N \text{.}
\end{align*}
One advantage of identity \eqref{moments} is that when $k$ is fixed and $N \to \infty$, $Z_k(\cdot)$ on the right hand side is a sum of 
$2^k$ terms and hence the multinomial expansion of $Z_k(\cdot)^N$ gives a number of terms, which grows 
like a polynomial in $N$. This leads to the following result. 
\begin{Corollary}
\label{cor}
\begin{equation}
\label{log_moments}
\lim_{N \to \infty} \frac{1}{N} \log \mathbb{E} \Big[ Z_N(\beta) \Big]^k = \max_{p_{\sigma}} \Big\{ - \sum_{\sigma \in \{-1, 1\}^k} p_{\sigma} \log p_{\sigma} 
+ \frac{\beta^2}{4} \sum_{\sigma, \sigma' \in\{-1, 1\}^k} p_{\sigma} p_{\sigma'} (\sigma \cdot \sigma')^2 \Big\} \text{,}
\end{equation}
where the maximum is taken over $2^k$ non-negative real numbers $p_\sigma$ indexed by $\sigma \in \{-1, 1\}^k$ such that $\sum_{\sigma} p_{\sigma} = 1$.
\end{Corollary}
The parameters $p_{\sigma}$, $\sigma \in \{-1, 1\}^k$ in \eqref{log_moments} above are most naturally interpreted as a probability mass function over 
$\{-1, 1\}^k$ and $ - \sum_{\sigma} p_{\sigma} \log p_{\sigma}$ as its information entropy (up to a multiplicative constant of $\log 2$). So let us define the following notation. 
\begin{Definition}
\label{pmf_entropy}$ $
\begin{itemize}
\item $\Lambda_k$ is the set of all the probability mass functions on $\{-1, 1\}^k$.
\[
\Lambda_k := \Big\{( p_{\sigma})_{\sigma \in \{-1, 1\}^k} \ : \ p_{\sigma} \geq 0 \ \forall \ \sigma \in \{-1, 1\}^k \ , \ \sum_{\sigma \in \{-1, 1\}^k} p_{\sigma} = 1 \Big\}
\]
\item The entropy of $\mathcal{L} \in \Lambda_k$, denoted by  $\mathcal{E}(\mathcal{L})$, is 
\[
\mathcal{E}(\mathcal{L}) := - \sum_{\sigma \in \{-1, 1\}^k} \mathcal{L}(\sigma) \log_2 \mathcal{L}(\sigma)
\]
\end{itemize}
\end{Definition}
In view of the above definition equation \eqref{log_moments} can also be written as
\begin{equation}
\label{log_moments2}
\lim_{N \to \infty} \frac{1}{N} \log \mathbb{E} \Big[ Z_N(\beta) \Big]^k = \max_{\mathcal{L} \in \Lambda_k} \Big\{ \log 2 \ \mathcal{E}(\mathcal{L})
+ \frac{\beta^2}{4} E^{\mathcal{L}} \Big( V \cdot V' \Big)^2 \Big\} \text{,}
\end{equation}
where $V \cdot V'$ is the scalar product of two independent identically-distributed random vectors $V$ and $V' \ \in \{-1, 1\}^k$ with probability mass function $\mathcal{L}$ 
and $E^{\mathcal{L}}$ is the expectation with respect to randomness over $\{-1, 1\}^k$.

The rest of this article is organised as follows. In Section 2 we discuss some basic properties of the maximisation problem 
\eqref{log_moments}-\eqref{log_moments2}. In Section 3 we prove Proposition \ref{prop} and Corollary \ref{cor}. And in Section 4 
we show how our results relate to those of Talagrand in \cite{4}.
\section{Some Discussion}
For shortness let us denote the quantity we wish to maximise in \eqref{log_moments} by $F(\cdot)$:
\[
F \Big( (p_{\sigma})_{\sigma \in \{-1, 1\}^k} \Big) := - \sum_{\sigma \in \{-1, 1\}^k} p_{\sigma} \log p_{\sigma} + \frac{\beta^2}{4} 
\sum_{\sigma , \sigma' \in \{-1, 1\}^k} p_{\sigma} p_{\sigma '} (\sigma \cdot \sigma')^2 \text{.}
\]
Firstly, let us observe that the maximisers of $F(\cdot)$ do not lie on the boundary of the hyperplane $\Lambda_k$ 
given by $\Big\{( p_{\sigma})_{\sigma \in \{-1, 1\}^k} \in \Lambda_k \ : \ p_{\sigma_0} = 0$ for some $\sigma_0 \in \{-1, 1\}^k \Big\}$ 
(this is not true in the degenerate case $\beta = \infty$ as we shall see later).
\begin{Proposition}[Maximisers of $F(\cdot)$ are local maxima]
\label{boundary}
Suppose $(p_{\sigma}) \in \Lambda_k$ and there exists $\sigma_0 \in \{-1, 1\}^k$ such that $p_{\sigma_0} = 0$. Then there exists 
$(q_{\sigma}) \in \Lambda_k$ such that $F\big( (q_{\sigma})\big) > F\big( (p_{\sigma})\big)$.
\end{Proposition}
\begin{proof}
Let $(p_{\sigma}) \in \Lambda_k$ be such that $p_{\sigma_0} = 0$ for some $\sigma_0 \in \{-1, 1\}^k$ and let $\sigma_1 \in \{-1, 1\}^k$ 
be such that $p_{\sigma_1} > 0$. 

For a small $\epsilon$ ($\epsilon < p_{\sigma_1}$) define  $q_{\sigma_0} := \epsilon$, $q_{\sigma_1} := p_{\sigma_1} - \epsilon$ and $q_{\sigma} := p_{\sigma}$ for all 
$\sigma \neq \sigma_0, \sigma_1$. Then
\[
F\big( (q_{\sigma})\big) - F\big( (p_{\sigma})\big) = - \epsilon \log \epsilon + O(\epsilon)
\]
and thus for $\epsilon$ sufficiently small $F\big( (q_{\sigma})\big) - F\big( (p_{\sigma})\big) > 0$.
\end{proof}
Thus any maximiser of $F(\cdot)$ must lie inside the hyperplane $\Lambda_k$ and in principle can be found using the method 
of Lagrangian multipliers.
\begin{Proposition}[Symmetry of the maximiser]
\label{symmetry}
Let $(p_{\sigma})_{\sigma \in \{-1, 1\}^k}$ be a maximiser of $F(\cdot)$. Then $p_{\sigma} = p_{- \sigma}$ for all $\sigma \in \{-1, 1\}^k$.
\end{Proposition}
\begin{proof}
Suppose that there exists a maximiser of $F(\cdot)$ such that $p_{\sigma_0} \neq p_{- \sigma_0}$ for some $\sigma_0 \in \{-1, 1\}^k$. 
Define $q_{\sigma} := p_{\sigma}$ for all $\sigma \neq \pm \sigma_0$, $q_{\sigma_0} := q_{- \sigma_0} := (p_{\sigma_0} + p_{- \sigma_0})/2$.

It is easy to check that to maximise $-a \log a - b \log b$ subject to $a$, $b \geq 0$, $a + b = c$ one needs to take $a = b  = c/2$ . So 
\[
-q_{\sigma_0} \log q_{\sigma_0} -q_{- \sigma_0} \log q_{- \sigma_0} > -p_{\sigma_0} \log p_{\sigma_0} -p_{- \sigma_0} \log p_{- \sigma_0} 
\]
and therefore
\begin{align*}
F((p_{\sigma})) = &\Big[- \sum_{\sigma \neq \pm \sigma_0} p_{\sigma} \log p_{\sigma} \Big] -p_{\sigma_0} \log p_{\sigma_0} -p_{- \sigma_0} \log p_{- \sigma_0}\\ 
& \quad + \frac{\beta^2}{4} \Big[ \sum_{\sigma, \sigma' \neq \pm \sigma_0} p_{\sigma} p_{\sigma'} (\sigma \cdot \sigma')^2 + 
\big[ 2(p_{\sigma_0} + p_{- \sigma_0}) \sum_{\sigma \neq \pm \sigma_0} p_{\sigma} (\sigma \cdot \sigma_0)^2 \big] + (p_{\sigma_0} + p_{- \sigma_0})^2 k^2 \Big]\\
= &\Big[- \sum_{\sigma \neq \pm \sigma_0} q_{\sigma} \log q_{\sigma} \Big] -p_{\sigma_0} \log p_{\sigma_0} -p_{- \sigma_0} \log p_{- \sigma_0}\\ 
& \quad + \frac{\beta^2}{4} \Big[ \sum_{\sigma, \sigma' \neq \pm \sigma_0} q_{\sigma} q_{\sigma'} (\sigma \cdot \sigma')^2 + 
\big[ 2(q_{\sigma_0} + q_{- \sigma_0}) \sum_{\sigma \neq \pm \sigma_0} q_{\sigma} (\sigma \cdot \sigma_0)^2 \big] + (q_{\sigma_0} + q_{- \sigma_0})^2 k^2 \Big]\\
< &\Big[- \sum_{\sigma \neq \pm \sigma_0} q_{\sigma} \log q_{\sigma} \Big] -q_{\sigma_0} \log q_{\sigma_0} -q_{- \sigma_0} \log q_{- \sigma_0}\\ 
& \quad + \frac{\beta^2}{4} \Big[ \sum_{\sigma, \sigma' \neq \pm \sigma_0} q_{\sigma} q_{\sigma'} (\sigma \cdot \sigma')^2 + 
\big[ 2(q_{\sigma_0} + q_{- \sigma_0}) \sum_{\sigma \neq \pm \sigma_0} q_{\sigma} (\sigma \cdot \sigma_0)^2 \big] + (q_{\sigma_0} + q_{- \sigma_0})^2 k^2 \Big]\\
= &F((q_{\sigma}))
\end{align*}
which contradicts the maximising property of $(p_{\sigma})_{\sigma \in \{-1, 1\}^k}$.
\end{proof}
It is natural to look at $F(\mathcal{L})$ as a weighted sum of the entropy of $\mathcal{L}$ and the expected value of the 
squared scalar product of two independent vectors with p.m.f. $\mathcal{L}$. Then it is easy to find the maximisers of $F(\cdot)$ 
in the two extreme cases when $\beta = 0$ and when $\beta \to \infty$.

In the first case
\[
\lim_{N \to \infty} \frac{1}{N} \log \mathbb{E} \Big[ Z_N(0) \Big]^k =  \log 2 \ \max_{\mathcal{L} \in \Lambda_k} 
\mathcal{E}(\mathcal{L}) = \log 2 \ k \text{,}
\]
since the entropy is known to be uniquely maximised by the uniform distribution (that is, $\mathcal{L}(\sigma) = 2^{-k} \ \forall \sigma$). 

In the second case 
\[
\lim_{N \to \infty} \frac{1}{N} \log \mathbb{E} \Big[ Z_N(\beta) \Big]^k \sim \frac{\beta^2}{4} \max_{\mathcal{L} \in \Lambda_k} E^{\mathcal{L}} \Big( V \cdot V' \Big)^2 
\qquad \text{ as } \beta \to \infty \text{.}
\]
Then since $|V \cdot V'| \leq k$ with equality if and only if $V = \pm V'$ it follows that $E^{\mathcal{L}} ( V \cdot V' )^2 \leq k^2$ with equality if and only if 
$\mathcal{L}$ is concentrated on $\pm \sigma_0$ for any $\sigma_0 \in \{-1, 1\}^k$. So 
\[
\lim_{N \to \infty} \frac{1}{N} \log \mathbb{E} \Big[ Z_N(\beta) \Big]^k \sim \frac{\beta^2}{4} \max_{\mathcal{L} \in \Lambda_k} 
E^{\mathcal{L}} \Big( V \cdot V' \Big)^2 = \frac{\beta^2}{4}k^2 \qquad \text{ as } \beta \to \infty \text{,}
\]
where the maximising p.m.f.'s  satisfy $\mathcal{L}(\sigma_0) + \mathcal{L}(- \sigma_0) = 1$ for any $\sigma_0 \in \{-1, 1\}^k$ and of these p.m.f.'s, 
for the maximisers of $F(\cdot)$, we would prefer the ones with $\mathcal{L}(\sigma_0) = \mathcal{L}(- \sigma_0) = \frac{1}{2}$ because of Proposition \ref{symmetry}.

Thus as $\beta$ varies from $0$ to $\infty$ we would expect the $F(\cdot)$-maximising p.m.f.'s to vary from the uniform distribution 
(lying in the middle of the hyperplane $\Lambda_k$) to the p.m.f.'s  satisfying $\mathcal{L}(\sigma_0) = \mathcal{L}(- \sigma_0) = \frac{1}{2}$ 
for a $\sigma_0 \in \{-1, 1\}^k$ (which lie on the boundary of $\Lambda_k$). Or, in other words, we expect the $F(\cdot)$-maximising 
p.m.f.'s to vary from completely dispersed at $\beta = 0$ to concentrated at $\pm \sigma_0$ at $\beta = \infty$.

This seems to make some physical sense since at infinite temperature ($\beta = 0$) the entropy of a physical system is at its maximum, while 
at $0$ temperature ($\beta = \infty$) the system is frozen. However the exact physical interpretation of the p.m.f.'s 
$(p_{\sigma})_{\sigma \in \{-1, 1\}^k} \in \Lambda_k$ is not clear to us.
\section{Proofs}
In this section we present the proofs of Proposition \ref{prop} and Corollary \ref{cor}
\begin{proof}[Proof of Proposition \ref{prop}]
We have that
\begin{align*}
\mathbb{E} \Big[ Z_N(\beta) \Big]^k &= \mathbb{E} \Big[ 
\sum_{\sigma^1, \cdots, \sigma^k \in \{-1, 1\}^N} \exp \Big\{ \frac{\beta}{\sqrt{N}} \sum_{1 \leq i < j \leq N} g_{i j} 
(\sigma^1_i \sigma^1_j + \cdots + \sigma^k_i \sigma^k_j ) \Big\}\Big]\\
&= \sum_{\sigma^1, \cdots, \sigma^k \in \{-1, 1\}^N} \exp \Big\{
\frac{\beta^2}{2N} \sum_{1 \leq i < j \leq N}(\sigma^1_i \sigma^1_j + \cdots + \sigma^k_i \sigma^k_j )^2 \Big\} \text{,}
\end{align*}
where $\sigma^1$, $ \cdots$, $\sigma^k$ are independent copies of $ \sigma$ and we have taken the expectation of a 
log-normal random variable. Then expanding the square and swapping the order of summation gives the following identity: 
\begin{align*}
\sum_{1 \leq i < j \leq N}(\sigma^1_i \sigma^1_j + \cdots + \sigma^k_i \sigma^k_j )^2 &= 
\sum_{1 \leq i < j \leq N} \Big( k + 2 \sum_{1 \leq u < v \leq k} \sigma^u_i \sigma^u_j \sigma^v_i \sigma^v_j \Big)\\
&= k \frac{N(N-1)}{2} + 2 \Big( \sum_{1 \leq u < v \leq k} \sum_{1 \leq i < j \leq N} 
\sigma^u_i \sigma^v_i \sigma^u_j \sigma^v_j \Big)\\
&= k \frac{N(N-1)}{2} + \Big( \sum_{1 \leq u < v \leq k} \Big[
(\sigma^u_1 \sigma^v_1 + \cdots + \sigma^u_N \sigma^v_N)^2 - N \Big]  \Big)\\
&= k \frac{N(N-1)}{2} - N \frac{k(k-1)}{2} + \sum_{1 \leq u < v \leq k} (\sigma^u_1 \sigma^v_1 + \cdots + \sigma^u_N \sigma^v_N)^2\\
&=\frac{kN^2}{2} - \frac{Nk^2}{2} + \sum_{1 \leq u < v \leq k} (\sigma^u_1 \sigma^v_1 + \cdots + \sigma^u_N \sigma^v_N)^2 \text{.}
\end{align*}
Therefore
\begin{align*}
\mathbb{E} \Big[ Z_N(\beta) \Big]^k &= \sum_{\sigma^1, \cdots, \sigma^k \in \{-1, 1\}^N} \exp \Big\{
\frac{\beta^2}{2N} \sum_{1 \leq i < j \leq N}(\sigma^1_i \sigma^1_j + \cdots + \sigma^k_i \sigma^k_j )^2 \Big\} \\
&= \mathrm{e}^{\frac{\beta^2 kN}{4} - \frac{\beta^2 k^2}{4}} \sum_{\sigma^1, \cdots, \sigma^k \in \{-1, 1\}^N} 
\exp \Big\{\frac{\beta^2}{2N} \sum_{1 \leq u < v \leq k}(\sigma^u_1 \sigma^v_1 + \cdots + \sigma^u_N \sigma^v_N )^2 \Big\}\\
&= \mathrm{e}^{\frac{\beta^2 kN}{4} - \frac{\beta^2 k^2}{4}} \sum_{\sigma_1, \cdots, \sigma_N \in \{-1, 1\}^k} \mathbb{E} \Big(
\exp \Big\{ \frac{\beta}{\sqrt{N}} \sum_{1 \leq u < v \leq k} g'_{u v} (\sigma^u_1 \sigma^v_1 + \cdots + \sigma^u_N \sigma^v_N ) \Big\}\Big)\\
&= \mathrm{e}^{\frac{\beta^2 kN}{4} - \frac{\beta^2 k^2}{4}} \mathbb{E} \Big[ Z_k(\beta \sqrt{\frac{k}{N}}) \Big]^N \text{,}
\end{align*}
where $\sigma_1, \cdots, \sigma_N \in \{-1, 1\}^k$ can be thought of as the rows of the $N \times k$ matrix whose columns are 
$\sigma^1, \cdots, \sigma^k$ and $g'_{u v }$ are  independent standard normal random variables.
\end{proof}
Recall that Corollary \ref{cor} stated that 
\[
\lim_{N \to \infty} \frac{1}{N} \log \mathbb{E} \Big[ Z_N(\beta) \Big]^k = \max_{p_{\sigma}} \Big\{ - \sum_{\sigma \in \{-1, 1\}^k} p_{\sigma} \log p_{\sigma} 
+ \frac{\beta^2}{4} \sum_{\sigma, \sigma' \in\{-1, 1\}^k} p_{\sigma} p_{\sigma'} (\sigma \cdot \sigma')^2 \Big\} \text{,}
\]
where the maximum is taken over $2^k$ non-negative real numbers $p_\sigma$ indexed by $\sigma \in \{-1, 1\}^k$ such that $\sum_{\sigma} p_{\sigma} = 1$.
\begin{proof}[Proof of Corollary \ref{cor}]
In this proof we shall only work with the space $\{-1, 1\}^k$ and we shall use the following simplified notation to make formulae more compact: 
\begin{itemize}
\item $\sum_{\sigma}$ and $\prod_{\sigma}$ will stand for the sum and the product over all $\sigma \in \{-1, 1\}^k$. Likewise 
$\sum_{\sigma , \sigma'}$ will stand for the sum over all $\sigma$, $\sigma' \in \{-1, 1\}^k$.
\item $\sum_{i<j}$ will stand for the sum over all $i, j \in \mathbb{N}$ such that $1 \leq i < j \leq k$.
\item $\sum_{i_{\sigma}}$ and $\max_{i_{\sigma}}$ will stand for the sum and the maximum over all combinations of non-negative 
integers $i_{\sigma}$, $\sigma \in \{-1, 1\}^k$ satisfying $\sum_{\sigma \in \{-1, 1\}^k} i_{\sigma} = N$.
\item $\max_{p_{\sigma}}$ will stand for the maximum over all combinations of non-negative reals $p_{\sigma}$, $\sigma \in \{-1, 1\}^k$ satisfying 
$\sum_{\sigma \in \{-1, 1\}^k} p_{\sigma} = 1$.
\end{itemize}
Starting with \eqref{moments} and applying the multinomial expansion to $Z_k(\cdot)^N$ we get
\begin{align*}
\mathbb{E} \Big[ Z_N(\beta) \Big]^k &= \mathrm{e}^{\frac{\beta^2 k}{4} N - \frac{\beta^2 k^2}{4}} \mathbb{E} \Big[ Z_k\Big(\beta \sqrt{\frac{k}{N}}\Big) \Big]^N\\
&= \mathrm{e}^{\frac{\beta^2 k}{4} N - \frac{\beta^2 k^2}{4}} \mathbb{E} \Big[ \sum_{\sigma} \exp \Big\{ \frac{\beta}{\sqrt{N}} 
\sum_{i < j} g_{i j} \sigma_i \sigma_j \Big\} \Big]^N\\
&= \mathrm{e}^{\frac{\beta^2 k}{4} N - \frac{\beta^2 k^2}{4}} \mathbb{E} \Big[ \sum_{i_{\sigma}} \frac{N!}{\prod_{\sigma}i_{\sigma}!} 
\prod_{\sigma} \exp \Big\{ i_{\sigma} \frac{\beta}{\sqrt{N}} \sum_{i < j} g_{i j} \sigma_i \sigma_j \Big\} \Big] \text{,}
\end{align*}
where the first summation in the last line is over all the combinations of non-negative integers $i_{\sigma}$, $\sigma \in \{-1, 1\}^k$ such that $\sum_{\sigma \in \{-1, 1\}^k} i_{\sigma} = N$. 
Then moving the product into the exponential and taking the expectation (of a log-normal) gives 
\begin{align*}
&\mathrm{e}^{\frac{\beta^2 k}{4} N - \frac{\beta^2 k^2}{4}} \mathbb{E} \Big[ \sum_{ i_{\sigma}} \frac{N!}{\prod_{\sigma}i_{\sigma}!} 
\prod_{\sigma} \exp \Big\{ i_{\sigma} \frac{\beta}{\sqrt{N}} \sum_{i < j} g_{i j} \sigma_i \sigma_j \Big\} \Big]\\
=&\mathrm{e}^{\frac{\beta^2 k}{4} N - \frac{\beta^2 k^2}{4}} \mathbb{E} \Big[ \sum_{i_{\sigma}} \frac{N!}{\prod_{\sigma}i_{\sigma}!} 
\exp \Big\{ \frac{\beta}{\sqrt{N}} \sum_{i < j} g_{i j} \sum_{\sigma} i_{\sigma} \sigma_i \sigma_j \Big\} \Big]\\
 =&\mathrm{e}^{\frac{\beta^2 k}{4} N - \frac{\beta^2 k^2}{4}}\sum_{i_{\sigma}} \frac{N!}{\prod_{\sigma}i_{\sigma}!} 
\exp \Big\{ \frac{\beta^2}{2N} \sum_{i < j} \Big(\sum_{\sigma} i_{\sigma} \sigma_i \sigma_j\Big)^2 \Big\} \text{.}
\end{align*}
Thus we have shown so far that 
\begin{equation}
\label{eq1}
\mathbb{E} \Big[ Z_N(\beta) \Big]^k = \sum_{i_{\sigma}} \frac{N!}{\prod_{\sigma}i_{\sigma}!} \exp \Big\{ \frac{\beta^2 k}{4}N - 
\frac{\beta^2 k^2}{4}+ \frac{\beta^2}{2N} \sum_{i < j} \Big(\sum_{\sigma} i_{\sigma} \sigma_i \sigma_j\Big)^2 \Big\} \text{.}
\end{equation}
Let us now prove that
\begin{equation}
\label{eq2}
\lim_{N \to \infty} \frac{1}{N} \log \mathbb{E} \Big[ Z_N(\beta) \Big]^k = \max_{p_{\sigma}} \Big\{ - \sum_{\sigma} p_{\sigma} \log p_{\sigma} + \frac{\beta^2 k}{4} 
+ \frac{\beta^2}{2} \sum_{i < j} \Big( \sum_{\sigma} p_{\sigma} \sigma_i \sigma_j \Big)^2 \Big\} \text{,}
\end{equation}
where the maximum is taken over all vectors $(p_{\sigma})_{\sigma \in \{-1, 1\}^k}$ with non-negative real entries such that $\sum_{\sigma \in \{-1, 1\}^k} p_{\sigma} = 1$.

To deal with the multinomial coefficient in \eqref{eq1} we are going to use the following well-known form of Stirling's approximation:
\begin{equation}
\label{stirling}
 n^{n + 1/2} \mathrm{e}^{-n} \leq n! \leq \mathrm{e} \ n^{n + 1/2} \mathrm{e}^{-n} \qquad \forall n \in \mathbb{N} \text{.}
\end{equation}
Let us begin with proving the upper bound of \eqref{eq2}. Observe that the summation $\sum_{i_{\sigma}}$ in \eqref{eq1} has 
$\binom{N + 2^k - 1}{2^k - 1}$ terms and thus 
\begin{align*}
\mathbb{E} \Big[ Z_N(\beta) \Big]^k &\leq \binom{N + 2^k - 1}{2^k - 1} \max_{i_{\sigma}} \Big\{ 
\frac{N!}{\prod_{\sigma}i_{\sigma}!} \exp \Big\{ \frac{\beta^2 k}{4}N - \frac{\beta^2 k^2}{4} 
+ \frac{\beta^2}{2N} \sum_{i < j} \Big(\sum_{\sigma} i_{\sigma} \sigma_i \sigma_j\Big)^2 \Big\} \Big\}\\
&\leq \big( N + 2^k \big)^{2^k} \max_{i_{\sigma}} \Big\{ \frac{\mathrm{e} \ N^{N+1/2} \mathrm{e}^{-N}}{
\prod_{\sigma : i_{\sigma} > 0}i_{\sigma}^{i_{\sigma} + 1/2} \mathrm{e}^{-i_{\sigma}} }\exp \Big\{ \frac{\beta^2 k}{4}N - \frac{\beta^2 k^2}{4}\\
& \qquad\qquad\qquad\qquad\qquad\qquad\qquad\qquad\qquad\qquad\qquad + \frac{\beta^2}{2N} \sum_{i < j} \Big(\sum_{\sigma} i_{\sigma} \sigma_i \sigma_j\Big)^2 \Big\} \Big\}\\
&\leq \big( N + 2^k \big)^{2^k} \mathrm{e} \ \max_{i_{\sigma}} \Big\{ \frac{1}{\prod_{\sigma : i_{\sigma} > 0}i_{\sigma}^{1/2}}\Big\}
\max_{i_{\sigma}} \Big\{ \frac{N^{N+1/2}}{\prod_{\sigma}i_{\sigma}^{i_{\sigma}}}\exp \Big\{ \frac{\beta^2 k}{4}N - \frac{\beta^2 k^2}{4}\\
& \qquad\qquad\qquad\qquad\qquad\qquad\qquad\qquad\qquad\qquad\qquad + \frac{\beta^2}{2N} \sum_{i < j} \Big(\sum_{\sigma} i_{\sigma} \sigma_i \sigma_j\Big)^2 \Big\} \Big\}\\ \text{.}
\end{align*}
Since $\max_{i_{\sigma}} \Big\{ \frac{1}{\prod_{\sigma : i_{\sigma} > 0}i_{\sigma}^{1/2}}\Big\} \leq 1$ we have
\begin{align*}
\mathbb{E} \Big[ Z_N(\beta) \Big]^k &\leq \big( N + 2^k \big)^{2^k} \mathrm{e} \ 
\max_{i_{\sigma}} \Big\{ \frac{N^{N+1/2}}{\prod_{\sigma}i_{\sigma}^{i_{\sigma}}}\exp \Big\{ \frac{\beta^2 k}{4}N - \frac{\beta^2 k^2}{4}\\
& \qquad\qquad\qquad\qquad\qquad\qquad\qquad\qquad\qquad\qquad\qquad + \frac{\beta^2}{2N} \sum_{i < j} \Big(\sum_{\sigma} i_{\sigma} \sigma_i \sigma_j\Big)^2 \Big\} \Big\}\\
&\leq \big( N + 2^k \big)^{2^k} \mathrm{e} \ \max_{p_{\sigma}} \Big\{
\frac{N^{N+1/2}}{\prod_{\sigma} (p_{\sigma}N)^{p_{\sigma}N}} \exp \Big\{ \frac{\beta^2 k}{4}N - \frac{\beta^2 k^2}{4}\\ 
& \qquad\qquad\qquad\qquad\qquad\qquad\qquad\qquad\qquad\qquad\qquad + \frac{\beta^2 N}{2} \sum_{i < j} \Big(\sum_{\sigma} p_{\sigma} \sigma_i \sigma_j\Big)^2 \Big\} \Big\} \text{,}
\end{align*}
where $\max_{p_{\sigma}}$ is taken over all vectors $(p_{\sigma})_{\sigma \in \{-1, 1\}^k}$ with non-negative entries 
such that\newline $\sum_{\sigma \in \{-1, 1\}^k} p_{\sigma} = 1$ and the last inequality follows from substituting $i_{\sigma} = p_{\sigma}N$. Moving everything inside the exponential gives
\begin{align*}
\mathbb{E} \Big[ Z_N(\beta) \Big]^k &\leq \big( N + 2^k \big)^{2^k} \mathrm{e} \ \max_{p_{\sigma}} 
\exp\Big\{ N \log N + \frac{1}{2} \log N - \sum_{\sigma} \big( p_\sigma N \log p_{\sigma} \\ 
& \qquad\qquad\qquad + p_{\sigma} N \log N \big) + \frac{\beta^2 k}{4}N - \frac{\beta^2 k^2}{4} + \frac{\beta^2 N}{2} \sum_{i < j} \Big(\sum_{\sigma} p_{\sigma} \sigma_i \sigma_j\Big)^2\Big\}\\
&= \big( N + 2^k \big)^{2^k} \mathrm{e} \ \max_{p_{\sigma}} \exp\Big\{ N\Big( -\sum_{\sigma} p_{\sigma} \log p_{\sigma} + \frac{\beta^2 k}{4}
+ \frac{\beta^2}{2} \sum_{i < j} \Big(\sum_{\sigma} p_{\sigma} \sigma_i \sigma_j\Big)^2 \Big)\\
& \qquad\qquad\qquad + \frac{1}{2} \log N  - \frac{\beta^2 k^2}{4} \Big\} \text{.}
\end{align*}
Taking the logarithm of the above inequality and dividing it by $N$ gives 
\begin{align*}
\frac{1}{N} \log \mathbb{E} \Big[ Z_N(\beta) \Big]^k &\leq \max_{p_{\sigma}} \Big\{ - \sum_{\sigma} p_{\sigma} \log p_{\sigma} + \frac{\beta^2 k}{4} 
+ \frac{\beta^2}{2} \sum_{i < j} \Big( \sum_{\sigma} p_{\sigma} \sigma_i \sigma_j \Big)^2 \Big\}\\
& \qquad + \frac{1}{N} \log \Big( \big( N + 2^k \big)^{2^k} \mathrm{e} \Big) 
+ \frac{1}{N} \Big( \frac{1}{2} \log N - \frac{\beta^2 k^2}{4} \Big) \text{.}
\end{align*}
Taking $\limsup_{N \to \infty}$ gives
\[
\limsup_{N \to \infty} \frac{1}{N} \log \mathbb{E} \Big[ Z_N(\beta) \Big]^k \leq \max_{p_\sigma} \Big\{ - \sum_{\sigma} p_{\sigma} \log p_{\sigma} + \frac{\beta^2 k}{4} 
+ \frac{\beta^2}{2} \sum_{i < j} \Big( \sum_{\sigma} p_{\sigma} \sigma_i \sigma_j \Big)^2 \Big\} \text{.}
\] 
We shall now prove the lower bound of \eqref{eq2}. Let $p_{\sigma}^{\ast}$, $\sigma \in \{-1, 1\}^k$ be a vector of maximising values of 
\[
- \sum_{\sigma} p_{\sigma} \log p_{\sigma} + \frac{\beta^2 k}{4} + \frac{\beta^2}{2} \sum_{i < j} \Big( \sum_{\sigma} p_{\sigma} \sigma_i \sigma_j \Big)^2 \text{.}
\]
Then there exists a sequence of vectors $i_{\sigma}(N)$, $N \geq 2$ such that $\sum_{\sigma} i_{\sigma}(N) = N$ and 
\[
\frac{i_{\sigma}(N)}{N} \to p_{\sigma}^{\ast} \qquad \text{ as } N \to \infty \quad \forall \sigma \in \{ -1, 1\}^k
\]
\Big( E.g. one can take $i_{\sigma}(N) := \lfloor N p_{\sigma}^{\ast} \rfloor$ for all $\sigma \neq (1, \cdots, 1)$ and 
$i_{\sigma}(N) := N - \sum_{\sigma \neq (1, \cdots, 1)} i_{\sigma}(N)$ for $\sigma = (1, \cdots, 1)$\Big)\newline
Therefore for any $\epsilon > 0$ there exists $N_\epsilon$ such that for all $N \geq N_\epsilon$
\begin{align*}
&\max_{p_{\sigma}} \Big\{ - \sum_{\sigma} p_{\sigma} \log p_{\sigma} + \frac{\beta^2 k}{4} + \frac{\beta^2}{2} \sum_{i < j} \Big( \sum_{\sigma} p_{\sigma} \sigma_i \sigma_j \Big)^2 \Big\} \\
= &- \sum_{\sigma} p_{\sigma}^{\ast} \log p_{\sigma}^{\ast} + \frac{\beta^2 k}{4} + \frac{\beta^2}{2} \sum_{i < j} \Big( \sum_{\sigma} p_{\sigma}^{\ast} \sigma_i \sigma_j \Big)^2\\
\leq &- \sum_{\sigma} \frac{i_{\sigma}(N)}{N} \log \frac{i_{\sigma}(N)}{N} + \frac{\beta^2 k}{4} + \frac{\beta^2}{2} \sum_{i < j} \Big( \sum_{\sigma} 
\frac{i_{\sigma}(N)}{N} \sigma_i \sigma_j \Big)^2 + \epsilon \text{.}
\end{align*}
Then from identity \eqref{eq1} and inequality \eqref{stirling} we have
\begin{align*}
\mathbb{E} \Big[ Z_N(\beta) \Big]^k &\geq \frac{N!}{\prod_{\sigma}i_{\sigma}(N)!} \exp \Big\{ \frac{\beta^2 k}{4}N - \frac{\beta^2 k^2}{4} 
+ \frac{\beta^2}{2N} \sum_{i < j} \Big(\sum_{\sigma} i_{\sigma}(N) \sigma_i \sigma_j\Big)^2 \Big\}\\
&\geq  \frac{N^{N+1/2} \mathrm{e}^{-N}}{\mathrm{e}^{2^k}\prod_{\sigma : i_{\sigma}(N)>0}i_{\sigma}(N)^{i_{\sigma}(N) + 1/2} \mathrm{e}^{-i_{\sigma}(N)} }
\exp \Big\{ \frac{\beta^2 k}{4}N - \frac{\beta^2 k^2}{4} + \frac{\beta^2}{2N} \sum_{i < j} \Big(\sum_{\sigma} i_{\sigma}(N) \sigma_i \sigma_j\Big)^2 \Big\}\\
&\geq  \frac{N^{N+1/2}}{\mathrm{e}^{2^k} N^{2^{k-1}} \prod_{\sigma : i_{\sigma}(N)>0}i_{\sigma}(N)^{i_{\sigma}(N)}}
\exp \Big\{ \frac{\beta^2 k}{4}N - \frac{\beta^2 k^2}{4} + \frac{\beta^2}{2N} \sum_{i < j} \Big(\sum_{\sigma} i_{\sigma}(N) \sigma_i \sigma_j\Big)^2 \Big\} \text{.}
\end{align*}
Moving everything into the exponential gives
\begin{align*}
\mathbb{E} \Big[ Z_N(\beta) \Big]^k &\geq \frac{1}{\mathrm{e}^{2^k} N^{2^{k-1}}} \exp \Big\{ N \log N + \frac{1}{2} \log N - 
\sum_{\sigma} i_{\sigma}(N) \log i_{\sigma}(N)\\
& \qquad\qquad\qquad + \frac{\beta^2 k}{4}N - \frac{\beta^2 k^2}{4} + \frac{\beta^2}{2N} \sum_{i < j} \Big(\sum_{\sigma} i_{\sigma}(N) \sigma_i \sigma_j\Big)^2 \Big\}\\
&= \frac{1}{\mathrm{e}^{2^k} N^{2^{k-1}}} \exp \Big\{ \frac{1}{2} \log N - 
N \sum_{\sigma} \frac{i_{\sigma}(N)}{N} \log \frac{i_{\sigma}(N)}{N} + \frac{\beta^2 k}{4}N - \frac{\beta^2 k^2}{4}\\
& \qquad\qquad\qquad + \frac{\beta^2 N}{2} \sum_{i < j} \Big(\sum_{\sigma} \frac{i_{\sigma}(N)}{N} \sigma_i \sigma_j\Big)^2 \Big\}
\end{align*}
Now fix $\epsilon > 0$. Taking the logarithm of the above inequality and dividing it by $N$ gives
\begin{align*}
\frac{1}{N} \log \mathbb{E} \Big[ Z_N(\beta) \Big]^k &\geq \frac{1}{N} \log \Big( \frac{1}{ \mathrm{e}^{2^k} N^{2^{k-1}}} \Big) + 
\frac{1}{2} \frac{1}{N} \log N - \frac{1}{N} \frac{\beta^2 k^2}{4}\\
& \qquad\qquad - \sum_{\sigma} \frac{i_{\sigma}(N)}{N} \log \frac{i_{\sigma}(N)}{N} + \frac{\beta^2 k}{4} + 
\frac{\beta^2}{2} \sum_{i < j} \Big(\sum_{\sigma} \frac{i_{\sigma}(N)}{N} \sigma_i \sigma_j\Big)^2 \\
&\geq \frac{1}{N} \log \Big( \frac{1}{ \mathrm{e}^{2^k} N^{2^{k-1}}} \Big) + \frac{1}{2} \frac{1}{N} \log N - \frac{1}{N} \frac{\beta^2 k^2}{4}\\
& \qquad\qquad - \sum_{\sigma} p_{\sigma}^{\ast} \log p_{\sigma}^{\ast} + \frac{\beta^2 k}{4} + \frac{\beta^2}{2} \sum_{i < j} \Big( \sum_{\sigma} p_{\sigma}^{\ast} \sigma_i \sigma_j \Big)^2 
- \epsilon \quad \forall N \geq N_{\epsilon} \text{.}
\end{align*}
Taking $\liminf_{N \to \infty}$ of the above inequality gives
\[
\liminf_{N \to \infty} \frac{1}{N} \log \mathbb{E} \Big[ Z_N(\beta) \Big]^k \geq - \sum_{\sigma} p_{\sigma}^{\ast} \log p_{\sigma}^{\ast} 
+ \frac{\beta^2 k}{4} + \frac{\beta^2}{2} \sum_{i < j} \Big( \sum_{\sigma} p_{\sigma}^{\ast} \sigma_i \sigma_j \Big)^2 - \epsilon \text{,}
\]
which holds true for all $\epsilon > 0$. Letting $\epsilon \searrow 0$ gives 
\[
\liminf_{N \to \infty} \frac{1}{N} \log \mathbb{E} \Big[ Z_N(\beta) \Big]^k \geq - \sum_{\sigma} p_{\sigma}^{\ast} \log p_{\sigma}^{\ast} 
+ \frac{\beta^2 k}{4} + \frac{\beta^2}{2} \sum_{i < j} \Big( \sum_{\sigma} p_{\sigma}^{\ast} \sigma_i \sigma_j \Big)^2 \text{,}
\]
which finishes the proof of \eqref{eq2}.

Finally, Corollary \ref{cor} follows from \eqref{eq2} by the simple observation that 
\begin{align*}
\sum_{i<j}\Big( \sum_{\sigma} p_{\sigma}^{\ast} \sigma_i \sigma_j \Big)^2 = &\sum_{i<j} \sum_{\sigma, \sigma'} p_{\sigma}^{\ast} 
p_{\sigma'}^{\ast} \sigma_i \sigma_j \sigma_i' \sigma_j'\\
= &\sum_{\sigma, \sigma'} p_{\sigma}^{\ast} p_{\sigma'}^{\ast} \sum_{i<j} \sigma_i \sigma_j \sigma_i' \sigma_j'\\
= &\sum_{\sigma, \sigma'} p_{\sigma}^{\ast} p_{\sigma'}^{\ast} \Big(\frac{(\sigma \cdot \sigma')^2 - k}{2} \Big)\\
= &\frac{1}{2} \sum_{\sigma, \sigma'} p_{\sigma}^{\ast} p_{\sigma'}^{\ast} (\sigma \cdot \sigma')^2 - \frac{k}{2} \text{.}
\end{align*}
\end{proof}
\section{Comparison of \eqref{log_moments} with the result of M. Talagrand}
In \cite{4} Michel Talagrand considered the $p$-spin model with the partition function 
\[
\hat{Z}_N(\beta) = \sum_{\sigma \in \{-1, 1\}^N} \exp \Big\{ \sum_{p \geq 1} \frac{\beta_p}{N^{(p-1)/2}} 
\sum_{i_1, \cdots, i_p}^N g_{i_1 \cdots i_p} \sigma_{i_1} \cdots \sigma_{i_p} \Big\} \text{,}
\]
where $g_{i_1 \cdots i_p}$'s are i.i.d. standard normal random variables and $\beta_p$'s are some real numbers 
satisfying $\sum_{p \geq 2} \beta_p p^2 < \infty$. 

We are only interested in the case when $\beta_p = 0$ for $p \neq 2$ and $\beta_2 = \beta$. In such a setup
\begin{align}
\label{partition2}
\hat{Z}_N(\beta) &= \sum_{\sigma \in \{-1, 1\}^N} \exp \Big\{ \frac{\beta}{\sqrt{N}} \sum_{i,j = 1}^N g_{i j} \sigma_i \sigma_j \Big\}\nonumber\\
&= \sum_{\sigma \in \{-1, 1\}^N} \exp \Big\{ \frac{\beta}{\sqrt{N}} \sum_{i = 1}^N g_{i i} + 
\frac{\beta}{\sqrt{N}} \sum_{1 \leq i < j \leq N}^N (g_{i j} + g_{j i}) \sigma_i \sigma_j \Big\}\nonumber\\
&\stackrel{d}{=} \exp \{ \beta g_{1 1} \} \sum_{\sigma \in \{-1, 1\}^N} \exp \Big\{\frac{\beta \sqrt{2}}{\sqrt{N}} \sum_{1 \leq i < j \leq N}^N g_{i j} \sigma_i \sigma_j \Big\}\text{.}
\end{align}
In Theorem 9.4 in \cite{4} Talagrand gives the following result: 
\begin{align}
\label{log_moments_talagrand}
\lim_{N \to \infty} \frac{1}{kN} \log \mathbb{E} \Big[ \hat{Z}_N(\beta)\Big]^k = \sup_{q \in [0,1]} \Big\{ \log 2 + 
\frac{1}{2}(\beta^2 - 2 \beta^2 q) &+ \frac{1}{k} \log \mathbb{E} \cosh^k(\beta \sqrt{2q} Z)\nonumber\\ 
& \qquad\qquad - \frac{1}{2}(k - 1) \beta^2 q^2\Big\} \text{,}
\end{align}
where $Z \stackrel{d}{=} N(0,1)$. In order to compare this result with our Corollary \ref{cor} we should first replace $\hat{Z}_N$ with $Z_N$. 
From \eqref{partition2} we see that 
\[
\lim_{N \to \infty} \frac{1}{N} \log \mathbb{E} \Big[ \hat{Z}_N(\beta)\Big]^k = \lim_{N \to \infty} \frac{1}{N} \log \mathbb{E} \Big[ Z_N(\sqrt{2}\beta)\Big]^k
\]
and hence \eqref{log_moments_talagrand} can be rewritten as
\begin{align}
\label{log_moments_talagrand2}
\lim_{N \to \infty} \frac{1}{N} \log \mathbb{E} \Big[ Z_N(\beta)\Big]^k &= 
\sup_{q \in [0,1]} \Big\{ k \log 2 + \frac{1}{4} k (\beta^2 - 2 \beta^2 q^2) + \log \mathbb{E} \cosh^k (\beta \sqrt{q} Z)\nonumber\\ 
&\qquad\qquad\qquad\qquad\qquad\qquad\qquad\qquad\quad - \frac{1}{4}k(k-1) \beta^2 q^2 \Big\}\nonumber\\
&= \sup_{q \in [0, \frac{\beta^2}{2}]} \Big\{ -k(k-1)\frac{q^2}{\beta^2} + \log H(q) - qk + \frac{1}{4}k \beta^2\Big\} \text{,}
\end{align}
where $H(q)$ is the quantity which will feature a lot in this section and which satisfies the following identities:
\begin{align}
\label{h}
H(q) &:= \sum_{\sigma \in \{-1, 1\}^k} \mathrm{e}^{q (\sigma \cdot \sigma_0)^2} \text{ (for any choice of } \sigma_0 \text{)}\nonumber\\
&= \sum_{i = 0}^k \binom{k}{i} \mathrm{e}^{(2i - k)^2q}\nonumber\\
&= \mathbb{E} \Big( \mathrm{e}^{\sqrt{2q}Z} + \mathrm{e}^{-\sqrt{2q}Z} \Big)^k \text{ , where } Z \stackrel{d}{=} N(0,1)\nonumber\\
&= 2^k \mathbb{E}\cosh^k(\sqrt{2q} Z)\text{.}
\end{align}
The function $H : [0, \infty) \to \mathbb{R}$ is important and before we proceed any further let us give a list of some of its basic properties. 
\begin{Proposition}[Properties of $H(q)$]$ $
\label{properties}
\begin{enumerate}
\item $H(0) = 2^k$, $H'(0) = k 2^k$
\item $\frac{\partial}{\partial q} \big( \frac{H'(q)}{H(q)} \big) = \frac{\partial^2}{\partial q^2} \big( \log H(q) \big) \geq 0 \ \ \forall q \in [0, \infty)$ \quad (in other words, $H$ is log-convex)
\item $k = \frac{H'(0)}{H(0)} \leq \frac{H'(q)}{H(q)} \leq \lim_{q \to \infty} \frac{H'(q)}{H(q)} = k^2$ \quad $\forall q \in [0, \infty)$ 
\end{enumerate}
\end{Proposition}
\begin{proof}
$(i)$ Using the first equation of \eqref{h} 
\[
H(0) = \sum_{\sigma \in \{-1, 1\}^k} 1 = 2^k \text{,}
\]
and for any choice of $\sigma_0 \in \{-1, 1\}^k$
\begin{align*}
H'(0) &= \sum_{\sigma \in \{-1, 1\}^k} (\sigma \cdot \sigma_0)^2 = \sum_{\sigma \in \{-1, 1\}^k} \sum_{i,j=1}^k 
\sigma^i \sigma_0^i \sigma^j \sigma_0^j\\
&= \sum_{\sigma \in \{-1, 1\}^k} (k + 2 \sum_{i < j} \sigma_0^i \sigma_0^j \sigma^i \sigma^j)\\
&= k2^k + 2 \sum_{i < j} \sigma_0^i \sigma_0^j \sum_{\sigma \in \{-1, 1\}^k} \sigma^i \sigma^j\\
&= k2^k \text{.}
\end{align*}
$(ii)$ 
\[
\frac{\partial}{\partial q} \Big( \frac{H'(q)}{H(q)} \Big) = \frac{H''(q)H(q) - H'(q)^2}{H(q)^2} \text{.}
\]
It is then sufficient to show that $H''(q)H(q) - H'(q)^2 \geq 0$ for all $q \geq 0$. Note that from the first equation of \eqref{h} for an arbitrary $ \sigma_0 \in \{-1, 1\}^k$
\begin{align*}
H''(q)H(q) - H'(q)^2 &= \Big[ \sum_{\sigma}(\sigma \cdot \sigma_0)^4 \mathrm{e}^{q(\sigma \cdot \sigma_0)^2}\Big] 
\Big[ \sum_{\sigma'} \mathrm{e}^{q(\sigma' \cdot \sigma_0)^2}\Big]\\
&\qquad\qquad\qquad\qquad - \Big[ \sum_{\sigma}(\sigma \cdot \sigma_0)^2 \mathrm{e}^{q(\sigma \cdot \sigma_0)^2}\Big] 
\Big[ \sum_{\sigma'}(\sigma' \cdot \sigma_0)^2 \mathrm{e}^{q(\sigma' \cdot \sigma_0)^2}\Big]\\
&= \sum_{\sigma, \sigma'} \Big[ (\sigma \cdot \sigma_0)^4 - (\sigma \cdot \sigma_0)^2(\sigma' \cdot \sigma_0)^2\Big] 
\mathrm{e}^{q(\sigma \cdot \sigma_0)^2 + q(\sigma' \cdot \sigma_0)^2}\\
&= \frac{1}{2}\sum_{\sigma, \sigma'} \Big[ (\sigma \cdot \sigma_0)^4 - 2(\sigma \cdot \sigma_0)^2(\sigma' \cdot \sigma_0)^2 
+ (\sigma' \cdot \sigma_0)^4\Big] \mathrm{e}^{q(\sigma \cdot \sigma_0)^2 + q(\sigma' \cdot \sigma_0)^2}\\
&= \frac{1}{2}\sum_{\sigma, \sigma'} \Big[ (\sigma \cdot \sigma_0)^2 - (\sigma' \cdot \sigma_0)^2\Big]^2 
\mathrm{e}^{q(\sigma \cdot \sigma_0)^2 + q(\sigma' \cdot \sigma_0)^2} \geq 0 \text{.}
\end{align*}
$(iii)$ Note that from the second equation in \eqref{h} 
\[
\lim_{q \to \infty} \mathrm{e}^{- k^2 q} H(q) = 
\lim_{q \to \infty} \mathrm{e}^{- k^2 q} \sum_{i = 0}^k \binom{k}{i} \mathrm{e}^{(2i - k)^2 q} = 2
\]
and 
\[
\lim_{q \to \infty} \mathrm{e}^{- k^2 q} H'(q) = 
\lim_{q \to \infty} \mathrm{e}^{- k^2 q} \sum_{i = 0}^k (2i - k)^2 \binom{k}{i} \mathrm{e}^{(2i - k)^2 q} = 2 k^2 \text{.}
\]
Thus $\lim_{q \to \infty} \frac{H'(q)}{H(q)} = \lim_{q \to \infty} \frac{\mathrm{e}^{-k^2  q}H'(q)}{\mathrm{e}^{-k^2  q}H(q)} = k^2$ 
and since $\frac{H'(q)}{H(q)}$ is increasing in $q$ and $\frac{H'(0)}{H(0)} = k$ it follows that 
\[
k = \frac{H'(0)}{H(0)} \leq \frac{H'(q)}{H(q)} \leq \lim_{q \to \infty} \frac{H'(q)}{H(q)} = k^2 \qquad \forall q \geq 0 \text{.}
\]
\end{proof}
Now let us continue with Talagrand's result. Define 
\[
f(q) := -k(k-1) \frac{q^2}{\beta^2} + \log H(q) - qk + \frac{1}{4}k \beta^2 
\]
Equation \eqref{log_moments_talagrand2} said that 
\[
\lim_{N \to \infty} \frac{1}{N} \log \mathbb{E} \Big[ Z_N(\beta)\Big]^k = \sup_{q \in [0, \frac{\beta^2}{2}]} f(q) \text{.}
\] 
From Proposition \ref{properties} we get
\[
f'(0) = \frac{H'(0)}{H(0)} - k = 0
\]
and
\[
f'(\frac{\beta^2}{2}) = -k(k - 1) + \frac{H'(\frac{\beta^2}{2})}{H(\frac{\beta^2}{2})} - k \leq -k(k - 1) + k^2 - k = 0
\]
(in fact $f'(q) < 0$ for all $q \geq \frac{\beta^2}{2}$).

So $f(\cdot)$ is flat at $0$ and non-increasing at $\frac{\beta^2}{2}$ and 
hence $\sup_{q \in [0, \frac{\beta^2}{2}]} f(q)$ is attained at a local maximum of $f(\cdot)$. Let $S$ be the set of all the 
local extrema of $f(\cdot)$: 
\begin{align}
\label{set}
S &:= \Big\{ q \in [0, \frac{\beta^2}{2}] \ : \ f'(q) = 0 \Big\}\nonumber\\
&= \Big\{ q \in [0, \frac{\beta^2}{2}] \ : \ q = \frac{\beta^2}{2k(k-1)} \Big( \frac{H'(q)}{H(q)} - k \Big) \Big\} \text{.}
\end{align}
It is always the case that $0 \in S$ and it seems that in general $S$ has between $1$ and $3$ elements depending 
on the values of $\beta$ and $k$. Then 
\begin{align}
\label{talagrand3}
\lim_{N \to \infty} \frac{1}{N} \log \mathbb{E} \Big[ Z_N(\beta)\Big]^k &= \sup_{q \in [0, \frac{\beta^2}{2}]} f(q)\nonumber\\
&= \sup_{q \in S} f(q)\nonumber\\
&= \sup_{q \in S} \Big\{ -k(k-1) \frac{q^2}{\beta^2} + \log H(q) - qk + \frac{1}{4}k \beta^2 \Big\} \text{.}
\end{align}
Thus we have rewritten the result of M. Talagrand \eqref{log_moments_talagrand} in the form \eqref{talagrand3} which will be more suitable for us.

Let us now look at our result. In Corollary \ref{cor} we have shown that
\begin{equation}
\label{max_p}
\lim_{N \to \infty} \frac{1}{N} \log \mathbb{E} \Big[ Z_N(\beta)\Big]^k = \max_{(p_{\sigma)} \in \Lambda_k} \Big\{ F \big( (p_{\sigma}) \big) \Big\} \text{,}
\end{equation}
where $ \Lambda_k = \Big\{ (p_{\sigma})_{\sigma \in \{-1, 1\}^k} \ : \ p_{\sigma} \geq 0 \ \forall \ \sigma \in \{-1, 1\}^k  \ , \sum_{\sigma \in \{-1, 1\}^k} p_{\sigma} = 1 \Big\}$ and
\[
F \big( (p_{\sigma}) \big) = - \sum_{\sigma \in \{-1, 1\}^k} p_{\sigma} \log p_{\sigma} + \frac{\beta^2}{4} 
\sum_{\sigma , \sigma' \in \{-1, 1\}^k} p_{\sigma} p_{\sigma '} (\sigma \cdot \sigma')^2 \text{.}
\]
Proposition \ref{boundary} suggested that the maxima of $F(\cdot)$ are solutions of the corresponding Lagrangian equations. 
The Lagrangian for the maximisation problem \eqref{log_moments} is given by 
\[
\Psi \big( (p_{\sigma}), \lambda \big) := - \sum_{\sigma} p_{\sigma} \log p_{\sigma} + 
\frac{\beta^2}{4} \sum_{\sigma, \sigma'} p_{\sigma} p_{\sigma'} (\sigma \cdot \sigma')^2 + 
\lambda \big( \sum_{\sigma} p_{\sigma} - 1 \big) \text{,}
\]
where $\lambda$ is the Lagrange multiplier. The partial derivatives of $\Lambda$ are: 
\begin{align*}
\frac{\partial \Psi}{\partial p_{\sigma_0}} &= - \log p_{\sigma_0} - 1 + \frac{\beta^2}{2} \sum_{\sigma} p_{\sigma} 
(\sigma \cdot \sigma_0)^2 + \lambda \quad \forall \sigma_0 \text{,}\\
\frac{\partial \Psi}{\partial \lambda} &= \sum_{\sigma} p_{\sigma} - 1 \text{.}
\end{align*}
Equating them to $0$ gives:
\begin{align}
\label{lagrange0}
\log p_{\sigma_0} &= \frac{\beta^2}{2} \sum_{\sigma} p_{\sigma} (\sigma \cdot \sigma_0)^2 + \lambda - 1 \quad \forall \sigma_0 \text{,}\\
\sum_{\sigma} p_{\sigma} &= 1 \text{.}\nonumber
\end{align}
Or, equivalently, 
\begin{equation}
\label{lagrange}
p_{\sigma_0} = C \mathrm{e}^{\frac{\beta^2}{2} \sum_{\sigma} p_{\sigma} (\sigma \cdot \sigma_0)^2} \quad \forall \sigma_0 \text{,}
\end{equation}
where $C$ is the normalising constant. 

We are not sure how one would rigorously solve \eqref{lagrange0} - \eqref{lagrange} (and whether it is even reasonable to look for 
all the solutions of \eqref{lagrange0} - \eqref{lagrange}). Nevertheless we luckily managed to find values of $(p_{\sigma})$ that solve \eqref{lagrange} and 
that make $F((p_{\sigma}))$ match the expression \eqref{talagrand3} given by Talagrand and which therefore must be the maximisers of $F(\cdot)$. 
However we cannot tell whether we have found all such values.

Before we present these values of $(p_{\sigma})$ let us prove the following useful result.
\begin{Proposition} $ $\newline
\label{sum}
(i) 
\[
\sum_{\sigma \in \{-1, 1\}^k} (\sigma \cdot \sigma_0)^2 = k2^k \qquad \forall \sigma_0 \in \{-1, 1\}^k \text{,}
\]
(ii)
\begin{align}
\label{eq_sum}
\sum_{\sigma \in \{-1, 1\}^k} (\sigma \cdot \sigma_0)^2 \mathrm{e}^{q (\sigma \cdot \sigma_1)^2} = 
&\frac{H'(q) - k H(q)}{k^2 - k} (\sigma_0 \cdot \sigma_1)^2\nonumber\\ 
\qquad &+ \frac{k^2 H(q) - H'(q)}{k - 1} \qquad \forall \sigma_0, \sigma_1 \in \{-1, 1\}^k , \ \forall q \in \mathbb{R} \text{,}
\end{align}
where $H(\cdot)$ is the function defined in \eqref{h}.
\end{Proposition}
\begin{proof}
(i) From Proposition \ref{properties} (i) 
\[
\sum_{\sigma} (\sigma \cdot \sigma_0)^2 = H'(0) = k2^k \text{.}
\]
(ii) Firstly, note that the left and the right hand sides of \eqref{eq_sum} as functions of $q$ have infinite radii of convergence 
about the origin (being just linear combinations of exponentials). Thus, it is sufficient to prove that the derivatives of all orders 
at the origin of both side of \eqref{eq_sum} are the same. That is, we need to show that
\begin{align}
\label{eq_sum2}
\sum_{\sigma \in \{-1, 1\}^k} (\sigma \cdot \sigma_0)^2 (\sigma \cdot \sigma_1)^{2n} = 
&\frac{H^{(n+1)}(0) - k H^{(n)}(0)}{k^2 - k} (\sigma_0 \cdot \sigma_1)^2\nonumber\\ 
\qquad &+ \frac{k^2 H^{(n)}(0) - H^{(n+1)}(0)}{k - 1} \qquad \forall \sigma_0, \sigma_1 \in \{-1, 1\}^k , \ \forall n \geq 0 \text{.}
\end{align}
Now,
\begin{align*}
\sum_{\sigma} (\sigma \cdot \sigma_0)^2 (\sigma \cdot \sigma_1)^{2n} &= \sum_{\sigma} 
\sum_{i_1, i_2, j_1, \cdots, j_{2n} = 1}^k \sigma^{i_1} \sigma_0^{i_1} \sigma^{i_2} \sigma_0^{i_2} \sigma^{j_1} 
\sigma_1^{j_1} \cdots \sigma^{j_{2n}} \sigma_1^{j_{2n}}\\
&= \sum_{i_1, i_2, j_1, \cdots, j_{2n} = 1}^k \sigma_0^{i_1} \sigma_0^{i_2} \sigma_1^{j_1} \cdots \sigma_1^{j_{2n}} 
\sum_{\sigma} \sigma^{i_1} \sigma^{i_2} \sigma^{j_1} \cdots \sigma^{j_{2n}} \text{.}
\end{align*}
The sum $\sum_{\sigma} \sigma^{i_1} \sigma^{i_2} \sigma^{j_1} \cdots \sigma^{j_{2n}}$ is non-zero only when 
each element of the set $\{i_1, i_2, j_1, \cdots, j_{2n}\}$ is equal to exactly an odd number of other elements of this set. E.g.,
\[
i_1 = i_2 = j_1 = j_2 \ , \ j_3 = \cdots = j_{2n} \ , \ i_1 \neq j_3
\]
Or,
\[
i_1 = j_1 \ , \ i_2 = j_2 \ , \ j_3 = \cdots = j_{2n} \ , \ i_1 \neq i_2 \ , \ i_1 \neq j_3 \ , \ i_2 \neq j_3 \text{.}
\]
In this case $\sigma^{i_1} \sigma^{i_2} \sigma^{j_1} \cdots \sigma^{j_{2n}} = 1$, $\sum_{\sigma} \sigma^{i_1} \sigma^{i_2} \sigma^{j_1} \cdots \sigma^{j_{2n}} = 2^k$ 
and $\sigma_0^{i_1} \sigma_0^{i_2} \sigma_1^{j_1} \cdots \sigma_1^{j_{2n}} = \sigma_0^{i_1} \sigma_1^{i_1} \sigma_0^{i_2} \sigma_1^{i_2}$ or $1$ 
depending on whether $i_1 = i_2$ or not. Thus 
\begin{align}
\label{CD}
\sum_{\sigma} (\sigma \cdot \sigma_0)^2 (\sigma \cdot \sigma_1)^{2n} &= \sum_{i_1, i_2, j_1, \cdots, j_{2n} = 1}^k \sigma_0^{i_1} \sigma_0^{i_2} 
\sigma_1^{j_1} \cdots \sigma_1^{j_{2n}} \sum_{\sigma} \sigma^{i_1} \sigma^{i_2} \sigma^{j_1} \cdots \sigma^{j_{2n}}\nonumber\\
&= C_n (\sigma_0 \cdot \sigma_1)^2 + D_n
\end{align}
for some constants $C_n$ and $D_n$ that do not depend on $\sigma_0$ or $\sigma_1$. Letting $\sigma_0 = \sigma_1$ yields 
\begin{equation}
\label{CD1}
H^{(n+1)}(0) = \sum_{\sigma} (\sigma \cdot \sigma_0)^{2n+2} = C_n k^2 + D_n \text{.}
\end{equation}
Summing \eqref{CD} over all $\sigma_0 \in \{-1, 1\}^k$ yields 
\begin{equation}
\label{CD2}
k2^k  H^{(n)}(0) = k2^k \sum_{\sigma} (\sigma \cdot \sigma_1)^{2n} = k2^k C_n + 2^k D_n \text{.}
\end{equation}
Solving \eqref{CD1} and \eqref{CD2} for $C_n$ and $D_n$ gives 
\[
C_n = \frac{H^{(n+1)}(0) - k H^{(n)}(0)}{k^2 - k} \ , \ D_n = \frac{k^2 H^{(n)}(0) - H^{(n+1)}(0)}{k - 1} 
\]
as required.
\end{proof}
Having proved identity \eqref{eq_sum2} let us now try to solve \eqref{lagrange0} by substituting 
\[
p_{\sigma} = \frac{\mathrm{e}^{q (\sigma \cdot \sigma_1)^2}}{H(q)} \ , \ \sigma \in \{-1, 1\}^k
\]
for $q \geq 0$ and an arbitrary of choice of $\sigma_1 \in \{-1, 1\}^k$, which we guess to be the right form of solution. On the left hand side we have 
\[
q(\sigma_0 \cdot \sigma_1)^2 - \log H(q) \text{.}
\]
On the right hand side we have 
\begin{align*}
&\frac{\beta^2}{2} \sum_{\sigma} (\sigma \cdot \sigma_0)^2 \frac{\mathrm{e}^{q(\sigma \cdot \sigma_1)^2}}{H(q)} + \lambda - 1\\
= &\frac{\beta^2}{2 H(q)} \Big[ \frac{H'(q) - k H(q)}{k^2 - k} (\sigma_0 \cdot \sigma_1)^2 + \frac{k^2 H(q) - H'(q)}{k - 1} \Big] + \lambda - 1 \text{,} 
\end{align*}
using Proposition \ref{sum}. The two sides of \eqref{lagrange0} must equal for all $\sigma_0 \in \{-1, 1\}^k$ and $\lambda$ can take any real value. Thus it is necessary that 
the coefficients in front of $(\sigma_0 \cdot \sigma_1)^2$ are equal and all the remaining terms can be absorbed into $\lambda$. In other words, for 
$p_{\sigma} = \frac{\mathrm{e}^{q (\sigma \cdot \sigma_1)^2}}{H(q)}$ to solve \eqref{lagrange0} $q$ must satisfy 
\[
q = \frac{\beta^2}{2 H(q)} \Big[ \frac{H'(q) - k H(q)}{k^2 - k}\Big]= \frac{\beta^2}{2k(k-1)} \Big[ \frac{H'(q)}{H(q)} - k \Big] \text{,}
\]
which is exactly the condition satisfied by the values of $q$ in the set $S$ in \eqref{set}. Now,
\begin{align*}
F \Big( \big( \frac{\mathrm{e}^{q(\sigma \cdot \sigma_1)^2}}{H(q)} \big)_{\sigma} \Big) &= 
- \sum_{\sigma} \frac{\mathrm{e}^{q(\sigma \cdot \sigma_1)^2}}{H(q)} \log \Big( \frac{\mathrm{e}^{q(\sigma \cdot \sigma_1)^2}}{H(q)} \Big) 
+ \frac{\beta^2}{4} \sum_{\sigma, \sigma'} (\sigma \cdot \sigma')^2 \frac{\mathrm{e}^{q(\sigma \cdot \sigma_1)^2}}{H(q)} 
\frac{\mathrm{e}^{q(\sigma' \cdot \sigma_1)^2}}{H(q)}\\
&= - \frac{1}{H(q)} \sum_{\sigma} \mathrm{e}^{q(\sigma \cdot \sigma_1)^2} \big( q(\sigma \cdot \sigma_1)^2 - \log H(q)\big)\\
& \qquad\qquad + \frac{\beta^2}{4 H(q)^2} \sum_{\sigma} \mathrm{e}^{q(\sigma \cdot \sigma_1)^2} \sum_{\sigma'} 
(\sigma \cdot \sigma')^2 \mathrm{e}^{q(\sigma' \cdot \sigma_1)^2}\\
&= - \frac{H'(q)}{H(q)} q + \log H(q)\\ 
& \qquad\qquad + \frac{\beta^2}{4 H(q)^2} \sum_{\sigma} \mathrm{e}^{q(\sigma \cdot \sigma_1)^2} 
\Big[ \frac{H'(q) - k H(q)}{k^2 - k} (\sigma \cdot \sigma_1)^2 + \frac{k^2 H(q) - H'(q)}{k - 1} \Big]
\end{align*}
using \eqref{eq_sum} and the facts that $H(q) = \sum_{\sigma} \mathrm{e}^{q(\sigma \cdot \sigma_1)^2} $ and 
$H'(q) = \sum_{\sigma} (\sigma \cdot \sigma_1)^2 \mathrm{e}^{q(\sigma \cdot \sigma_1)^2}$ in the last equality.
Continuing the simplification we further get 
\begin{align*}
F \Big( \big( \frac{\mathrm{e}^{q(\sigma \cdot \sigma_1)^2}}{H(q)} \big)_{\sigma} \Big) &= - \frac{H'(q)}{H(q)}q + \log H(q)\\
& \qquad\qquad + \frac{\beta^2}{4 H(q)^2} \Big[ H'(q) \frac{H'(q) - k H(q)}{k^2 - k} + H(q) \frac{k^2 H(q) - H'(q)}{k - 1} \Big]\\
&= - \frac{H'(q)}{H(q)}q + \log H(q) + \frac{\beta^2}{4k(k-1)} \Big[ \frac{H'(q)^2}{H(q)^2}  - 2k \frac{H'(q)}{H(q)} + k^3 \Big]
\end{align*}
Then for any $q \in S$, $\frac{H'(q)}{H(q)} = \frac{2k(k-1)}{\beta^2}q + k$ and thus 
\begin{align*}  
F \Big( \big( \frac{\mathrm{e}^{q(\sigma \cdot \sigma_1)^2}}{H(q)} \big)_{\sigma} \Big) &= - \Big( 
\frac{2k(k-1)}{\beta^2}q + k \Big)q + \log H(q)\\ 
&\qquad\qquad + \frac{\beta^2}{4k(k-1)} \Big[  \Big( 
\frac{2k(k-1)}{\beta^2}q + k \Big)^2 - 2k  \Big( \frac{2k(k-1)}{\beta^2}q + k \Big) + k^3 \Big]\\
&= - \frac{k(k-1)}{\beta^2}q^2 - kq + \log H(q) + \frac{\beta^2}{4} \text{,}
\end{align*}
which maximised over $q \in S$ gives exactly Talagrand's expression \eqref{talagrand3}.
\section{Future Research}
It would be interesting to investigate the asymptotic behaviour of $\mathbb{E} Z_N(\beta)^{k(N)}$ as $N \to \infty$, where $k(N)$ grows 
with $N$. That could for example be useful in estimating the exponential moments of $Z_N(\beta)$.

Another direction for future research is to try to generalise \eqref{log_moments} and \eqref{log_moments2} to all $k \in \mathbb{R}$ 
(or at least to all $k \in \mathbb{R}^{+}$). Then as suggested in \cite{1} we can recover an alternative characterisation of the Parisi 
formula via the following limiting procedure:
\[
\lim_{N \to \infty} \frac{1}{N} \mathbb{E} \log Z_N(\beta) = \lim_{k \to 0} \frac{1}{k} \Big[
\lim_{N \to \infty} \frac{1}{N} \log \mathbb{E} Z_N(\beta)^k\Big] \text{.}
\]


\begin{thebibliography}{ww}
\bibitem[1]{1} M. Mezard, G. Parisi, M.A.  Virasoro
{\em Spin Glass Theory and Beyond},
World Scientific Lecture Notes in Physics, vol. 9, 1987
\bibitem[2]{1a} D. Panchenko
{\em The Sherrington-Kirkpatrick model},
Springer Monographs in Mathematics, 2013
\bibitem[3]{2} D. Sherrington, S. Kirkpatrick, 
{\em Solvable Model of a Spin-Glass},
Physical Review Letters, vol. 35, 1975
\bibitem[4]{3} M. Talagrand, 
{\em The Parisi Formula},
Annals of Mathematics, vol. 163(1), 2006
\bibitem[5]{4} M. Talagrand, 
{\em Large Deviations, Guerra's and A.S.S. Schemes, and the Parisi Hypothesis},
Journal of Statistical Physics, vol. 126, 2007
\bibitem[6]{5} M. Talagrand, 
{\em Mean Field Models for Spin Glasses, volume I: Basic Examples},
Springer-Verlag, 2011
\bibitem[7]{6} M. Talagrand, 
{\em Mean Field Models for Spin Glasses, volume II: Advanced Replica-Symmetry and Low Temperature},
Springer-Verlag, 2011
\bibitem[8]{7} M. Aizenman, J.L. Lebowitz, D. Ruelle, 
{\em Some Rigorous Results on the Sherrington-Kirkpatrick Spin Glass Model},
Communications in Mathematical Physics, vol. 112, 1987

\end{thebibliography}
\end{document}